\documentclass[reqno,10pt,centertags]{amsart}
\usepackage{,amsthm,amssymb}
\date{\today}
\usepackage{hyperref} 
\newcommand*{\mailto}[1]{\href{mailto:#1}{\nolinkurl{#1}}}
\newcommand{\arxiv}[1]{\href{http://arxiv.org/abs/#1}{arXiv:#1}}

\newcommand{\bbC}{{\mathbb{C}}}

\newcommand{\bbN}{{\mathbb{N}}}

\newcommand{\bbR}{{\mathbb{R}}}

\newcommand{\cB}{{\mathcal B}}

\newcommand{\cH}{{\mathcal H}}

\newcommand{\cJ}{{\mathcal J}}

\newcommand{\cN}{{\mathcal N}}

\newcommand{\beq}{\begin{equation}}
\newcommand{\enq}{\end{equation}}

\DeclareMathOperator{\tr}{tr}

\newcommand{\dott}{\,\cdot\,}

\newcommand{\no}{\notag}
\newcommand{\lb}{\label}
\newcommand{\f}{\frac}

\newcommand{\wti}{\widetilde}

\newcommand{\bi}{\bibitem}

\renewcommand{\ge}{\geqslant}

\newcommand{\fatcdot}{\textbf{\Large.}}

\let\geq\geqslant
\let\leq\leqslant

\makeatletter
\def\theequation{\@arabic\c@equation}

\allowdisplaybreaks 
\numberwithin{equation}{section}

\newtheorem{theorem}{Theorem}[section]

\newtheorem{lemma}[theorem]{Lemma}
\newtheorem{corollary}[theorem]{Corollary}

\newtheorem{hypothesis}[theorem]{Hypothesis}
\newtheorem{example}[theorem]{Example}

\theoremstyle{remark}
\newtheorem{remark}[theorem]{Remark}

\begin{document}
\title[On Weighted Hardy-Type Inequalities]{On Weighted Hardy-Type Inequalities} 

\author[C.\ Y.\ Chuah et al.]{Chian Yeong Chuah}
\address{Department of Mathematics, 
Baylor University, One Bear Place \#97328,
Waco, TX 76798-7328, USA}
\email{\mailto{Chian\_Chuah@baylor.edu}}

\author[]{Fritz Gesztesy}
\address{Department of Mathematics, 
Baylor University, One Bear Place \#97328,
Waco, TX 76798-7328, USA}
\email{\mailto{Fritz\_Gesztesy@baylor.edu}}
\urladdr{\url{https://www.baylor.edu/math/index.php?id=935340}}

\author[]{Lance L. Littlejohn}
\address{Department of Mathematics, 
Baylor University, One Bear Place \#97328,
Waco, TX 76798-7328, USA}
\email{\mailto{Lance\_Littlejohn@baylor.edu}}
\urladdr{\url{https://www.baylor.edu/math/index.php?id=53980}}

\author[]{Tao Mei}
\address{Department of Mathematics, 
Baylor University, One Bear Place \#97328,
Waco, TX 76798-7328, USA}
\email{\mailto{Tao\_Mei@baylor.edu}}
\urladdr{\url{https://www.baylor.edu/math/index.php?id=925779}}

\author[]{Isaac Michael}
\address{Department of Mathematics, 
Baylor University, One Bear Place \#97328,
Waco, TX 76798-7328, USA}
\email{\mailto{Isaac\_Michael@baylor.edu}}
\urladdr{\url{http://blogs.baylor.edu/isaac\_michael/}}

\author[]{Michael M.\ H.\ Pang}
\address{Department of Mathematics,
University of Missouri, Columbia, MO 65211, USA}
\email{\mailto{pangm@missouri.edu}}
\urladdr{\url{https://www.math.missouri.edu/people/pang}}

\date{\today}
\thanks{{\it Mathematical Inequalities $\&$~Applications} (to appear).}

\begin{abstract}
We revisit weighted Hardy-type inequalities employing an elementary ad hoc approach that yields explicit constants. We also discuss the infinite sequence of power weighted Birman--Hardy--Rellich-type inequalities and derive an operator-valued version thereof. 
\end{abstract}

\maketitle


\section{Introduction} \lb{s1}

To put the results derived in this paper into some perspective, we very briefly recall some of the history of Hardy's celebrated inequality. We will exclusively focus on the continuous case even though Hardy originally started to investigate the discrete case (i.e., sums instead of integrals). 

Hardy's inequality, in its primordial version, is of the form
\begin{equation}
\int_0^{\infty} dx \, |f'(x)|^2 \geq 4^{-1} \int_0^{\infty} dx \, x^{-2} |f(x)|^2, \quad f \in C_0^{\infty}((0,\infty)),  \lb{1.1a}
\end{equation}
with the constant $4^{-1}$ being optimal and the inequality being a strict one for $f \neq 0$. (This extends to all $f \in AC([0,R])$ for all $R>0$, $f' \in L^2((0,\infty); dx)$, with $f(0_+) =0$, but we will not dwell on this improvement right now.) 
Hardy's work on his celebrated inequality started in 1915, \cite{Ha15}, see also 
\cite{Ha19}--\cite{Ha25}, and the historical comments in 
\cite[Chs.~1, 3, App.]{KMP07}. Soon afterwards, Hardy also proved a weighted Hardy inequality (with power weights) of the form (cf.\ \cite{Ha28}, \cite[Sect.~9.8]{HLP88}),
\begin{align}
\begin{split} 
\int_0^b dx \, x^{\alpha} |f'(x)|^p \geq \bigg(\f{|\alpha - p +1|}{p}\bigg)^p \int_0^{\infty} dx \, x^{\alpha - p} |f(x)|^p,& \\ 
p \in [1,\infty), \;  \alpha \in \bbR, \; f \in C_0^{\infty}((0,\infty)).& \lb{1.1b} 
\end{split} 
\end{align}
Again, the constant $(|\alpha - p +1|/p)^p$ is optimal and the inequality is strict for $f \neq 0$. 

Equation \eqref{1.1b} represents just the tip of an iceberg of weighted inequalities of Hardy-type. More generally, modern treatments of this subject are devoted to weighted inequalities of the form
\begin{equation}
\bigg(\int_a^b dx \, v(x) |f'(x)|^p\bigg)^{1/p} \geq C_{p,q} \bigg(\int_a^b dx \, w(x) |f(x)|^q\bigg)^{1/q}, \quad 
f \in C_0^{\infty} ((a,b)),     \lb{1.1c} 
\end{equation} 
for appropriate $a, b \in \bbR \cup \{\pm \infty\}$, $a < b$, $p, q \in [1,\infty) \cup \{\infty\}$, and appropriate weight functions $0 \leq v, w \in L^1_{loc}((a,b); dx)$. Again, this extends to certain optimal spaces for $f$, far beyond 
$f \in C_0^{\infty} ((a,b))$. We refer to \cite[Chs.~2--5]{KMP07}, \cite[Chs.~1,2]{KPS17}, \cite[Ch.~1]{OK90}, and the extensive literature cited therein. In particular, we mention the following integral versions of the two-weighted Hardy-type inequality \eqref{1.1c} (the former is sometimes referred to as the differential version),
\begin{align}
\begin{split} 
\bigg(\int_a^b dx \, v(x) |F(x)|^p\bigg)^{1/p} 
\geq C_{p,q} \bigg(\int_a^b dx \, w(x) \bigg|\int_a^x dx' \, F(x')\bigg|^q\bigg)^{1/q},& \\
F \in C_0^{\infty} ((a,b)),&     \lb{1.1d} 
\end{split} 
\end{align}  
and its companion (or ``dual'') version 
\begin{align}
\begin{split} 
\bigg(\int_a^b dx \, v(x) |F(x)|^p\bigg)^{1/p} 
\geq C_{p,q} \bigg(\int_a^b dx \, w(x) \bigg|\int_x^b dx' \, F(x')\bigg|^q\bigg)^{1/q},& \\
F \in C_0^{\infty} ((a,b)).&     \lb{1.1e} 
\end{split} 
\end{align}  
We note that many authors make the additional assumption $F \geq 0$ in \eqref{1.1d}, \eqref{1.1e}. 

Before describing the results obtained in this paper in some detail, we pause for a moment to introduce our notation: We start by briefly summarizing essentials on Bochner integrability and associated vector-valued $L^p$-spaces. 
Regarding details of the Bochner integral we refer, for instance, to \cite[p.\ 6--21]{ABHN01},
\cite[Ch.~1]{CM97}, \cite[p.\ 44--50]{DU77}, \cite[p.\ 71--86]{HP85}, \cite[Sect.~4.2]{Ku78}, 
\cite[Ch.\ III]{Mi78}, \cite[Sect.\ V.5]{Yo80}. 
In particular, if $p\ge 1$, $(a,b) \subseteq \bbR$ is a finite or infinite interval, 
$0 \leq w \in L^1_{loc}((a,b); dx)$ is a weight function, and $\cB$ a Banach space, 
the symbol $L^p((a,b);wdx;\cB)$ denotes the set of equivalence classes of strongly measurable $\cB$-valued functions which differ at most on sets of Lebesgue measure zero, such that 
$\|f(\cdot)\|_{\cB}^p \in L^1((a,b); wdx)$. The
corresponding norm in $L^p((a,b);wdx;\cB)$ is given by
\begin{equation}
\|f\|_{L^p((a,b);wdx;\cB)} = \bigg(\int_{(a,b)} w(x) dx \, \|f(x)\|_{\cB}^p\bigg)^{1/p}
\end{equation}
and $L^p((a,b);wdx;\cB)$ is a Banach space.
In the special case $\cB = \bbC$, we omit $\bbC$ and just write $L^p((a,b); w dx)$, respectively, 
$L^p_{loc}((a,b); w dx)$, as usual.  If $\cH$ is a separable Hilbert space, then so is 
$L^2((a,b);wdx;\cH)$ (see, e.g., \cite[Subsects.\ 4.3.1, 4.3.2]{BW83}, \cite[Sect.\ 7.1]{BS87}). 

One recalls that by a result of Pettis \cite{Pe38}, if $\cB$ is separable, weak
measurability of $\cB$-valued functions implies their strong measurability.

A map $f:[c,d] \to \cB$ (with $[c,d] \subset (a,b)$) is called \textit{absolutely continuous 
on $[c,d]$}, denoted by $f \in AC([c,d]; \cB)$, if 
\begin{equation}
f(x)= f(x_0) + \int_{x_0}^x dt \, g(t), \quad x_0, x \in [c,d], 
\end{equation}
for some $g \in L^1((c,d);dx;\cB)$. In particular, $f$ is then 
strongly differentiable a.e.\ on $(c,d)$ and
\begin{equation}
f'(x) = g(x) \, \text{ for a.e.\ $x \in (c,d)$}.
\end{equation}
Similarly, $f:[c,d] \to \cB$ is called \textit{locally absolutely continuous}, denoted by 
$f \in AC_{loc}([c,d]; \cB)$, if $f \in AC([c',d']; \cB)$ on any closed subinterval $[c',d'] \subset (c,d)$. 

For $p \in [1,\infty)$, its H\"older conjugate index $p'$ is given in a standard manner by 
$p' = p/(p-1) \in (1, \infty) \cup \{\infty\}$.  
 
If $\cH$ represents a complex, separable Hilbert space, then $\cB(\cH)$ denotes the Banach space (the $C^*$-algebra) of bounded, linear operators defined on all of $\cH$, and $\cB_p(\cH)$ denote the $\ell^p$-based 
Schatten--von Neumann trace ideals, $p \in [1,\infty)$, with ${\tr}_{\cH} (T)$ abbreviating the trace of a trace 
class operator $T \in \cB_1(\cH)$. 

Finally, we are in a position to briefly describe the principal result of our paper in Section \ref{s2}. Assume that $- \infty \leq a < b \leq \infty$, $p \in [1,\infty)$, and 
suppose that $0 \leq w_1 \in AC_{loc}((a,b))$, $0 \leq [- w_1']$ a.e.~on $(a,b)$, 
$0 \leq w_2 \in L^1_{loc}((a,b); dx)$, and $[- w_1']^{1-p} w_2^p \in L^1_{loc}((a,b); dx)$. If $F \in C_0((a,b); \cB)$, then we prove that 
\begin{align} 
\begin{split} 
& \int_a^b dx \, w_1(x)^p [- w_1'(x)]^{1-p} w_2(x)^p \|F(x)\|_{\cB}^p   \\
& \quad \geq p^{-p} \int_a^b dx \, [- w_1'(x)] \bigg(\int_a^x dx' \, w_2(x') \|F(x')\|_{\cB}\bigg)^p.   \lb{1.1} 
\end{split}
\end{align}
Moreover, we prove the companion result with $\int_a^x dx' \dots$ replaced by $\int_x^b dx' \dots$\,. As an 
important special case of \eqref{1.1} one recovers the classical form of the power weighted Hardy inequality 
\begin{align}
\begin{split} 
\int_0^b dx \, \, x^{\alpha} \, \|F(x)\|^p_{\cB} \geq 
\bigg(\f{|\alpha + 1 - p|}{p}\bigg)^p \int_0^b dx \, x^{\alpha -p} \bigg(\int_0^x dx' \,  \|F(x')\|_{\cB}\bigg)^p,&   
 \lb{1.2} \\
0 < b \leq \infty, \; p \in [1,\infty), \; \alpha < p - 1.&  
\end{split} 
\end{align} 
As alluded to earlier, the constant $[(|\alpha - p + 1|)/p]^p$ on the right-hand side of \eqref{1.2} is best possible, and equality holds if and only if $F = 0$~a.e.~on $(0,b)$. 
After describing appropriate iterations of \eqref{1.2} (again, including the companion results with $\int_a^x dx' \dots$ replaced by $\int_x^b dx' \dots$), we also recover as a special case the entire infinite sequence of the power weighted Birman--Hardy--Rellich-type inequalities (cf.\ \cite[p.~48]{Bi66}, \cite{GLMW18}, \cite[pp.~83--84]{Gl66}) at the end of Section \ref{s2}, namely, 
\begin{align}
\begin{split} 
\int_0^b dx \, x^{\alpha} |f^{(n)}(x)|^p \geq \f{\prod_{j=1}^k |\alpha- jp +1|^p}{p^{kp}}
\int_0^b dx \, x^{\alpha - kp} |f^{(n-k)}(x)|^p,& \\ 
0 < b \leq \infty, \; p \in [1,\infty), \; 1 \leq k \leq n, \; n \in \bbN, \; \alpha \in \bbR, \; f \in C_0^{\infty}((0,b)).&    \lb{1.3} 
\end{split} 
\end{align}

Replacing the restrictive hypothesis $F \in C_0((a,b); \cB)$ by the finiteness condition of the left-hand side in \eqref{1.1}, and a detailed discussion of best possible constants in these inequalities are the principal subjects 
of Section \ref{s3}. 

Finally, in Section \ref{s4} we consider extensions of \eqref{1.2} and of the 
infinite sequence of Birman--Hardy--Rellich-type inequalities to the operator-valued context, extending some results of Hansen \cite{Ha09}. More specifically, assuming $F : (0,b) \to \cB(\cH)$ is a weakly measurable map satisfying  
$\|F(\dott)\|_{\cB_p(\cH)} \in L^p((0,b); x^{\alpha} dx)$, we derive the inequality
\begin{align}
& {\tr}_{\cH}\bigg(\int_0^b dx \, \, x^{\alpha} \, |F(x)|^p\bigg) \geq 
\bigg(\f{|\alpha - p + 1|}{p}\bigg)^p {\tr}_{\cH}\bigg(\int_0^b dx \, x^{\alpha -p} 
\bigg|\int_0^x dx' \, F(x')\bigg|^p\bigg),    \no \\
& \hspace*{5.8cm} 0 < b \leq \infty, \; p \in [1,\infty), \; \alpha < p - 1.      \lb{1.4}
\end{align} 
Again, the constant $[(|\alpha - p + 1|)/p]^p$ on the right-hand side of 
\eqref{1.4} is best possible, and equality holds if and only if $F = 0$~a.e.~on $(0,b)$.

Moreover, for $p \in [1,2]$, we remove the trace in inequality \eqref{1.4} as follows: Suppose that 
$F : (0,\infty) \to \cB(\cH)$ is a weakly measurable map satisfying 
$F(\dott) \geq 0$~a.e.~on $(0,\infty)$, and 
$\int_0^{\infty} dx \, x^{\alpha} F(x)^p \in \cB(\cH)$, then we derive the operator-valued inequality 
\begin{align}
\begin{split}
\int_0^{\infty} dx \, \, x^{\alpha} \, F(x)^p \geq 
\bigg(\f{|\alpha - p + 1|}{p}\bigg)^p \int_0^{\infty} dx \, x^{\alpha -p} 
\bigg(\int_0^x dx' \, F(x')\bigg)^p,&    \lb{1.5} \\
p \in [1,2], \; \alpha < p - 1.&
\end{split} 
\end{align} 
Once again, 
the constant $[(|\alpha - p + 1|)/p]^p$ on the right-hand side of \eqref{1.5} is best possible, and equality holds if and only if $F = 0$~a.e.~on $(0,\infty)$. We also derive the corresponding companion results with $\int_a^x dx' \dots$ replaced by $\int_x^b dx' \dots$\,. 

We emphasize that \eqref{1.4} and \eqref{1.5} with $\alpha = 0$ (and hence $p > 1$) were  proved by Hansen in  \cite{Ha09}.

\section{Weighted Hardy-Type Inequalities Employing an Ad Hoc Approach} \lb{s2}

In this section we derive weighted Hardy inequalities employing an elementary ad hoc approach.

We begin by deriving a weighted Hardy inequality for $\cB$-valued functions and hence make the 
following assumptions.

\begin{hypothesis} \lb{h2.1} 
Let $- \infty \leq a < b \leq \infty$, $p \in [1,\infty)$, and $0 \leq w_2 \in L^1_{loc}((a,b); dx)$.  \\[1mm]
$(i)$ Suppose that $0 \leq w_1 \in AC_{loc}((a,b))$, $0 \leq [- w_1']$ a.e.~on $(a,b)$, 
$[- w_1']^{1-p} w_2^p \in L^1_{loc}((a,b); dx)$. \\[1mm] 
$(ii)$ Suppose that $0 \leq w_1 \in AC_{loc}((a,b))$, $0 \leq w_1'$ a.e.~on $(a,b)$, 
$[w_1']^{1-p} w_2^p \in L^1_{loc}((a,b); dx)$. 
\end{hypothesis}

The principal result of this section then reads as follows:

\begin{theorem} \lb{t2.2}
Let $p \in [1,\infty)$, and suppose that $F \in C_0((a,b); \cB)$.  \\[1mm] 
$(i)$ Assume Hypothesis \ref{h2.1}\,$(i)$, then
\begin{align} 
\begin{split} 
& \int_a^b dx \, w_1(x)^p [- w_1'(x)]^{1-p} w_2(x)^p \|F(x)\|_{\cB}^p   \\
& \quad \geq p^{-p} \int_a^b dx \, [- w_1'(x)] \bigg(\int_a^x dx' \, w_2(x') \|F(x')\|_{\cB}\bigg)^p.   \lb{2.5} 
\end{split}
\end{align}
$(ii)$ Assume Hypothesis \ref{h2.1}\,$(ii)$, then
\begin{align} 
\begin{split} 
& \int_a^b dx \, w_1(x)^p [w_1'(x)]^{1-p} w_2(x)^p \|F(x)\|_{\cB}^p    \\
& \quad \geq p^{-p} \int_a^b dx \, w_1'(x) \bigg(\int_x^b dx' \, w_2(x') \|F(x')\|_{\cB}\bigg)^p.    \lb{2.6} 
\end{split}
\end{align}
\end{theorem}
\begin{proof}
It suffices to prove \eqref{2.5} and then hint at the analogous proof of \eqref{2.6}. Since
\begin{align}
& \f{d}{dx} \bigg(w_1(x) \bigg(\int_a^x dx' \, w_2(x') \|F(x')\|_{\cB}\bigg)^p\bigg) 
= w_1'(x) \bigg(\int_a^x dx' \, w_2(x') \|F(x')\|_{\cB}\bigg)^p    \no \\
& \quad + p w_1(x) \bigg(\int_a^x dx' \, w_2(x') \|F(x')\|_{\cB}\bigg)^{p-1} w_2(x) \|F(x)\|_{\cB}, 
\end{align}
one obtains 
\begin{align}
& \int_a^b dx \, \f{d}{dx} \bigg(w_1(x) \bigg(\int_a^x dx' \, w_2(x') \|F(x')\|_{\cB}\bigg)^p\bigg)   \no \\
& \quad = w_1(x) \bigg(\int_a^x dx' \, w_2(x') \|F(x')\|_{\cB}\bigg)^p\bigg|_{x=a}^b  \no \\
& \quad =  w_1(b) \bigg(\int_a^b dx' \, w_2(x') \|F(x')\|_{\cB}\bigg)^p   \no \\
& \quad = \int_a^b dx \, w_1'(x) \bigg(\int_a^x dx' \, w_2(x') \|F(x')\|_{\cB}\bigg)^p   \no \\
& \qquad + p \int_a^b dx \, w_1(x) \bigg(\int_a^x dx' w_2(x') \|F(x')\|_{\cB}\bigg)^{p-1} w_2(x) \|F(x)\|_{\cB}.
\lb{2.7} 
\end{align}
Here we used that by hypothesis, $0 \leq w_1$ is monotonically decreasing, and that 
the right-hand side of \eqref{2.7} exists employing $F \in C_0((a,b); \cB)$. Thus, with $p^{-1} + [p']^{-1} = 1$, an application of H\"older's inequality yields 
\begin{align}
& w_1(b) \bigg(\int_a^b dx' \, w_2(x') \|F(x')\|_{\cB}\bigg)^p + 
\int_a^b dx \, [- w_1'(x)] \bigg(\int_a^x dx' \, w_2(x') \|F(x')\|_{\cB}\bigg)^p   \no \\
& \quad = p \int_a^b dx \, w_1(x) \bigg(\int_a^x dx' \, w_2(x') \|F(x')\|_{\cB}\bigg)^{p-1} w_2(x) \|F(x)||_{\cB} 
\no \\
& \quad \leq p \bigg[\int_a^b dx \, [- w_1'(x)] \bigg(\int_a^x dx' \, w_2(x') \|F(x')\|_{\cB}\bigg)^p\bigg]^{1/p'}  \no \\
& \qquad \times \bigg[\int_a^b dx \, w_1(x)^p [- w_1'(x)]^{1-p} w_2(x)^p \|F(x)\|_{\cB}^p \bigg]^{1/p} .
\end{align}
In particular,
\begin{align}
& \int_a^b dx \, [- w_1'(x)] \bigg(\int_a^x dx' \, w_2(x') \|F(x')\|_{\cB}\bigg)^p   \no \\
& \quad \leq p \bigg[\int_a^b dx \, [- w_1'(x)] \bigg(\int_a^x dx' \, w_2(x') \|F(x')\|_{\cB}\bigg)^p\bigg]^{1/p'}  \no \\
& \qquad \times \bigg[\int_a^b dx \, w_1(x)^p [- w_1'(x)]^{1-p} w_2(x)^p \|F(x)\|_{\cB}^p \bigg]^{1/p} 
\end{align}
and hence
\begin{align}
\begin{split} 
& \bigg[\int_a^b dx \, [- w_1'(x)] \bigg(\int_a^x dx' \, w_2(x') \|F(x')\|_{\cB}\bigg)^p\bigg]^{1/p}    \\
& \quad \leq p \bigg[\int_a^b dx \, w_1(x)^p [- w_1'(x)]^{1-p} w_2(x)^p \|F(x)\|_{\cB}^p \bigg]^{1/p}, 
\end{split} 
\end{align}
completing the proof of item $(i)$. 

For the proof of item $(ii)$ one notes the identity 
\begin{align}
& \f{d}{dx} \bigg(w_1(x) \bigg(\int_x^b dx' \, w_2(x') \|F(x')\|_{\cB}\bigg)^p\bigg) 
= w_1'(x) \bigg(\int_x^b dx' \, w_2(x') \|F(x')\|_{\cB}\bigg)^p    \no \\
& \quad - p w_1(x) \bigg(\int_x^b dx' \, w_2(x') \|F(x')\|_{\cB}\bigg)^{p-1} w_2(x) \|F(x)\|_{\cB},   \lb{2.11} 
\end{align}
and then obtains upon integrating \eqref{2.11} with respect to $x$ from $a$ to $b$,
\begin{align}
& w_1(a) \bigg(\int_a^b dx' \, w_2(x') \|F(x')\|_{\cB}\bigg)^p 
+ \int_a^b dx \, w_1'(x) \bigg(\int_x^b dx' \, w_2(x') \|F(x')\|_{\cB}\bigg)^p   \no \\
& \quad = p \int_a^b dx \, w_1(x) w_2(x) \bigg(\int_x^b dx' \, w_2(x') \|F(x')\|_{\cB}\bigg)^{p-1} \|F(x)\|_{\cB}.
\end{align}
In particular,
\begin{align}
\begin{split} 
&  \int_a^b dx \, w_1'(x) \bigg(\int_x^b dx' \, w_2(x') \|F(x')\|_{\cB}\bigg)^p    \\
& \quad \leq p \int_a^b dx \, w_1(x) w_2(x) \bigg(\int_x^b dx' \, w_2(x') \|F(x')\|_{\cB}\bigg)^{p-1} \|F(x)\|_{\cB},
\end{split} 
\end{align}
and now one can repeat the H\"older inequality argument as in item $(i)$. (Alternatively, if $b < \infty$, one can 
also prove item $(ii)$ 
by the change of variable $x \mapsto a + (b - x)$, i.e., by reflecting the interval $(a, b)$ at its midpoint).
\end{proof}

We illustrate our general result with the following well-known special case, the power-weighted Hardy inequality. For pertinent references on inequalities \eqref{2.13}, \eqref{2.14} below, we recall, for instance, 
\cite[Theorem~1.2.1]{BEL15}, \cite{Ha28}, 
\cite[p.~245--246]{HLP88}, \cite[p.~23, 43]{KMP07}, \cite[p.~9--11]{KPS17}, \cite[Lemma~1.3]{OK90}, and the references therein. 

\begin{example} \lb{e2.3}
Let $p \in [1,\infty)$, $a=0$, $b \in(0, \infty) \cup \{\infty\}$, $w_2(x) = 1$, and suppose that the map 
$F : (0,\infty) \to \cB$ is weakly measurable satisfying 
$\|F(\dott)\|_{\cB} \in L^p((0,b); x^{\alpha} dx)$. \\[1mm]  
$(i)$ If $w_1(x) = |\alpha - p + 1|^{-1} x^{- |\alpha - p + 1|}$, $\alpha < p - 1$, then \eqref{2.5} reduces to the classical form 
\begin{equation}
\int_0^b dx \, \, x^{\alpha} \, \|F(x)\|^p_{\cB} \geq 
\bigg(\f{|\alpha - p + 1|}{p}\bigg)^p \int_0^b dx \, x^{\alpha -p} \bigg(\int_0^x dx' \,  \|F(x')\|_{\cB}\bigg)^p.    \lb{2.13} 
\end{equation} 
$(ii)$  If $w_1(x) = [|\alpha - p + 1|]^{-1} x^{|\alpha - p + 1|}$, $\alpha > p - 1$, then \eqref{2.6} reduces to the complementary classical form 
\begin{equation}
\int_0^b dx \, \, x^{\alpha} \, \|F(x)\|^p_{\cB} \geq 
\bigg(\f{|\alpha - p + 1|}{p}\bigg)^p \int_0^b dx \, x^{\alpha -p} \bigg(\int_x^b dx' \,  \|F(x')\|_{\cB}\bigg)^p.    \lb{2.14} 
\end{equation} 
In both cases $(i)$ and $(ii)$, the constant $[(|\alpha - p + 1|)/p]^p$ is best possible and equality holds if and only 
if $F = 0$~a.e.~on $(0,b)$. 
\end{example} 

The case $F \in C_0((0,b); \cB)$ in Example \ref{e2.3} is a corollary of Theorem \ref{t2.2} and optimality of the constants on the right-hand sides of \eqref{2.13}, \eqref{2.14}, and the fact that equality is only attained in the trivial case $F = 0$~a.e.~on $(0,b)$, is a classical result (see, e.g., \cite[Theorem~1.2.1]{BEL15}). The extension of Example \ref{e2.3} to the case $F \in L^p((0,b); x^{\alpha} dx; \cB)$, $p \in [1,\infty)$, follows along the lines in \cite[Theorem~1.14, Sects.~1.3, 1.5]{OK90}. We will briefly return to this issue after Theorem \ref{t3.4}

Iterating the weighted Hardy inequality yields the sequence of vector-valued Birman inequalities as follows.    
Consider the iterated Hardy-type operators,  
\begin{align}
(H_{-,1} F)(x) &= \int_a^x dt_1 \, F(t_1),    \no \\ 
(H_{-,\ell} F)(x) &= H_{-,1} \bigg(\int_a^{\, \fatcdot} dt_2 \, \cdots \, \int_a^{t_{\ell-1}} dt_{\ell} 
\, F(t_{\ell})\bigg)(x)    \no \\
&=\int_a^x dt_1 \int_a^{t_1} dt_2 \, \cdots \, \int_a^{t_{\ell-1}} dt_{\ell} \, F(t_{\ell}) \no \\
&= [(\ell - 1)!]^{-1}\int_a^x dt \, (x - t)^{\ell - 1} F(t), \quad \ell \in \bbN, \; \ell \geq 2,   \lb{2.15} \\
& \hspace*{2.55cm} F \in L^p((a,c); dx) \, \text{ for all $c \in (a,b)$,}   \no \\ 
(H_{+,1} F)(x) &= \int_x^b dt_1 \, F(t_1),    \no \\ 
(H_{+,\ell} F)(x) &= H_{+,1} \bigg(\int_{\, \fatcdot}^b dt_2 \, \cdots \, \int_{t_{\ell-1}}^b dt_{\ell} 
\, F(t_{\ell})\bigg)(x)    \no \\
&= \int_x^b dt_1 \int_{t_1}^b dt_2 \, \cdots \, \int_{t_{\ell-1}}^b dt_{\ell} \, F(t_{\ell})   \no \\
&= [(\ell - 1)!]^{-1}\int_x^b dt \, (x - t)^{\ell - 1} F(t), \quad \ell \in \bbN, \; \ell \geq 2,    \lb{2.16} \\
& \hspace*{2.55cm} F \in L^p((c,b); dx) \, \text{ for all $c \in (a,b)$.}   \no 
\end{align}
Applying \eqref{2.13} and \eqref{2.14} iteratively in the form (with $a=0$) 
\begin{align}
\begin{split} 
& \int_0^b dx \, x^{\gamma - p} [(H_{\mp,1} (G_{\ell}(\dott))(x)]^p  \\
& \quad \leq \bigg(\f{p}{|\gamma - p +1|}\bigg)^p \int_0^{b} dx \, x^{\gamma} G_{\ell}(x)^p, \quad p \in [1, \infty), 
\; \gamma \lessgtr p-1,   \lb{2.16a} 
\end{split} 
\end{align}
for appropriate $0 \leq G_{\ell} \in L^p((0,b);x^{\gamma} dx)$, $p \in [1,\infty)$, 
then yields for $F : (0,\infty) \to \cB$ a weakly measurable map satisfying 
$\|F(\dott)\|_{\cB} \in L^p((0,b); x^{\alpha} dx)$
\begin{align}
&  \int_0^b dx \, x^{\alpha} \|F(x)\|_{\cB}^p    \no \\
& \quad \geq \prod_{k=1}^{\ell} \bigg(\f{|\alpha - kp + 1|}{p}\bigg)^p \int_0^b dx \, x^{\alpha - \ell p} 
[(H_{\mp,\ell} \|F(\, \cdot \,)\|_{\cB})(x)]^p,   \lb{2.17} \\
& \hspace*{2.1cm} 0 < b \leq \infty, \; p \in [1,\infty), \; \alpha \lessgtr \begin{cases} p - 1, \\ \ell p - 1, 
\end{cases}  \ell \in \bbN.   \no 
\end{align} 
It is well-known that the constants in \eqref{2.13}, \eqref{2.14} and \eqref{2.17}  are all optimal and that, in fact, these inequalities are all strict unless $F = 0$ on $(0,b)$. 

Turning to the differential form of the iterated (integral) Hardy inequalities \eqref{2.17}, and adding appropriate boundary conditions for $F$ at both endpoints $a, b$, permits one to avoid the gap 
$(p-1, \ell p -1)$ for $\alpha$ in \eqref{2.17} as follows: Assuming $F \in C_0^{\infty}((a,b); \cB)$ for simplicity, and introducing 
\begin{align}
\begin{split} 
&\wti f(x) = \int_0^x dx' \|F(x')\|_{\cB}, \; x \in (0,b), \quad \wti f^{(n)}(a) = 0, \; n \in \bbN,    \\ 
&\wti g(x) = \int_x^b dx' \|F(x')\|_{\cB}, \; x \in (0,b), \quad \wti g^{(n)}(b) = 0, \; n \in \bbN,   \lb{2.20}
\end{split} 
\end{align}
inequalities \eqref{2.13} and \eqref{2.14} become 
\begin{align}
& \int_0^b dx \, x^{\alpha} \big[{\wti f}^{\prime}(x)\big]^p \geq \bigg(\f{|\alpha - p +1|}{p}\bigg)^p \int_0^b dx \, x^{\alpha - p} \wti f(x)^p, 
\quad \alpha < p -1,  \lb{2.21} \\
& \int_0^b dx \, x^{\alpha} \big[- \wti g'(x)\big]^p \geq 
\bigg(\f{|\alpha - p +1|}{p}\bigg)^p \int_0^b dx \, x^{\alpha - p} \wti g(x)^p, 
\quad \alpha > p -1.   \lb{2.22} 
\end{align}
As a special case one obtains 
\begin{align}
\begin{split} 
\int_0^b dx \, x^{\alpha} |f'(x)|^p \geq \bigg(\f{|\alpha - p +1|}{p}\bigg)^p \int_0^b dx \, x^{\alpha - p} |f(x)|^p,& \\ 
0 < b \leq \infty, \; p \in [1,\infty), \;  \alpha \in \bbR, \; f \in C_0^{\infty}((0,b)).& \lb{2.23} 
\end{split} 
\end{align}
Iterating \eqref{2.23} yields the well-known result 
\begin{align}
\begin{split} 
\int_0^b dx \, x^{\alpha} |f^{(n)}(x)|^p \geq \f{\prod_{j=1}^k |\alpha- jp +1|^p}{p^{kp}}
\int_0^b dx \, x^{\alpha - kp} |f^{(n-k)}(x)|^p,& \\ 
0 < b \leq \infty, \; p \in [1,\infty), \; 1 \leq k \leq n, \; n \in \bbN, \; \alpha \in \bbR, \; f \in C_0^{\infty}((0,b)).&    \lb{2.24} 
\end{split} 
\end{align}

For additional results on higher-order (overdetermined) Hardy-type inequalities see also 
\cite[Ch.~4]{KPS17}, \cite{Na02}, \cite{NS97}, \cite{NS99}.  

\section{More on Weighted Hardy-Type Inequalities} \lb{s3}

To remove the assumption $F \in C_0((0,b); \cB)$ in Theorem \ref{t2.2} and to take a closer look at the issue of best possible constants in the inequality, we next recall (a generalization of) a celebrated 1969 result due to Talenti \cite{Ta69}, Tomaselli \cite{To69}, and shortly afterwards by Chisholm and Everitt \cite{CE70/71} and Muckenhoupt \cite{Mu72}, followed by Chisholm, Everitt, and Littlejohn \cite{CEL99}. 
For exhaustive textbook presentations we refer, for instance, to \cite[Sect.~1.2]{BEL15}, 
\cite[Sect.~5.3]{Da95}, \cite[Sect.~2.2]{EE04}, \cite[Sects.~9.8, 9.9]{HLP88}, \cite[Chs.~3, 4]{KMP07}, 
\cite[Chs.~1, 3, 4]{KPS17}, \cite[Sects.~1.1--1.3, 1.5, 1.6, 1.10]{OK90}. 

In addition to $H_{\mp,1}$ in \eqref{2.15}, \eqref{2.16}, we now also introduce the generalized (weighted) Hardy operators as follows.

\begin{hypothesis} \lb{h3.1}
Let $- \infty \leq a < b \leq \infty$ and $p \in [1,\infty)$. \\[1mm] 
$(i)$ Assume that $v$ and $w$ are weight functions satisfying $v, w$ measurable on $(a,b)$, 
$v > 0$, $w > 0$~a.e.\ on $(a,b)$. \\[1mm]
$(ii)$ Suppose that $\phi_{\mp}, \psi_{\mp}$ satisfy for all $c \in (a,b)$, 
\begin{align}
& \text{$0 < \phi_{\mp}$~a.e.\ on $(a,b)$, 
$0 < \psi_{\mp}$~a.e.\ on $(a,b)$},   \no \\
&\phi_- \in L^{p} ((c,b); v dx), \quad \psi_- \in L^{p'} \big((a,c); w^{-p'/p}dx\big),    \\
&\phi_+ \in L^{p} ((a,c); v dx), \quad \psi_+ \in L^{p'} \big((c,b); w^{-p'/p}dx\big).  \no 
\end{align}
\end{hypothesis}

Given Hypothesis \ref{h3.1} we introduce 
\begin{align}
& (H_{-,\phi_-, \psi_-} F)(x) = \phi_-(x) \int_a^x dx' \, \psi_-(x') F(x'),  \quad x \in (a,b),   \\
& \hspace*{3.35cm} F \in L^p((a,c); w dx) \, \text{ for all $c \in (a,b)$,}   \no \\ 
& (H_{+,\phi_+, \psi_+} F)(x) = \phi_+(x) \int_x^b dx' \, \psi_+(x') F(x'),  \quad x \in (a,b),    \\
& \hspace*{3.32cm} F \in L^p((c,b); w dx) \, \text{ for all $c \in (a,b)$.}   \no  
\end{align} 
In particular, $H_{\mp,1,1} = H_{\mp,1}$.

The following result, Theorem \ref{t3.2}, is well-known and a special case of more general situations recorded in the literature. For instance, we refer to \cite{Br78}, \cite{GKPW03}, 
\cite[p.~38--40]{KMP07}, \cite[Theorem 2.3]{KPS17} (after specializing to the case $\varphi_1=\psi_1=1$), 
and \cite[Theorem~1.14, Lemma~5.4 in Ch.~1]{OK90} (choosing $q=p$ in their results). 

\begin{theorem} \lb{t3.2} 
Assume Hypothesis \ref{h3.1}\,$(i)$. \\[1mm]
$(i)$ There exists a constant $C_- \in (0, \infty)$ such that
\begin{equation}
C_- \bigg(\int_a^b dx \, w(x) F(x)^p\bigg)^{1/p} \geq \bigg(\int_a^b dx \, v(x) [(H_{-,1} F)(x)]^p \bigg)^{1/p}   \lb{3.22}
\end{equation}
for all $F$ measurable on $(a,b)$ and $F \geq 0$~a.e.\ on $(a,b)$, if and only if 
\begin{equation}
A_- := \sup_{c \in (a,b)} \bigg(\int_c^b dx \, v(x)\bigg)^{1/p} 
\bigg(\int_a^c dx \, w(x)^{- p'/p}\bigg)^{1/p'} < \infty.    \lb{3.23}
\end{equation}
$($If $p=1$ and hence $p'=\infty$, the second factor in the right-hand side of \eqref{3.23} is interpreted as 
$\|1/w\|_{L^\infty((a,c);dx)}$.$)$   
Moreover, the smallest constant $C_{0,-} \in (0,\infty)$ in \eqref{3.22} satisfies 
\begin{align}
\begin{split} 
& A_- \leq C_{0,-} \leq p^{1/p} (p')^{1/p'} A_-, \quad p \in (1,\infty),   \\
& C_{0,-} = A_-, \quad p = 1.
\end{split} 
\end{align}
$(ii)$ There exists a constant $C_+ \in (0, \infty)$ such that
\begin{equation}
C_+ \bigg(\int_a^b dx \, w(x) F(x)^p\bigg)^{1/p} \geq \bigg(\int_a^b dx \, v(x) [(H_{+,1} F)(x)]^p \bigg)^{1/p}   \lb{3.25}
\end{equation}
for all $F$ measurable on $(a,b)$ and $F \geq 0$~a.e.\ on $(a,b)$, if and only if
\begin{equation}
A_+ := \sup_{c \in (a,b)} \bigg(\int_a^c dx \, v(x)\bigg)^{1/p} 
\bigg(\int_c^b dx \, w(x)^{- p'/p}\bigg)^{1/p'}     \lb{3.26}
< \infty. 
\end{equation}
$($If $p=1$ and hence $p'=\infty$, the second factor in the right-hand side of \eqref{3.26} is interpreted as 
$\|1/w\|_{L^\infty((c,b);dx)}$.$)$ 
Moreover, the smallest constant $C_{0,+} \in (0,\infty)$ in \eqref{3.25} satisfies 
\begin{align}
\begin{split} 
& A_+ \leq C_{0,+} \leq p^{1/p} (p')^{1/p'} A_+, \quad p \in (1, \infty),    \\
& C_{0,+} = A_+ \quad p = 1.
\end{split} 
\end{align}
\end{theorem}

We emphasize that items $(i)$ and $(ii)$ in Theorem \ref{t3.2} do not exclude the trivial case where the left-hand sides of \eqref{3.22} and \eqref{3.25} are infinite. 

We also note that Theorem \ref{t3.2} naturally extends to $p = \infty$, but as we will not use this in this note we omit further details (cf.\ \cite[Sect.~1.5]{OK90}). Moreover, \cite[Sects.~1.3, 1.5]{OK90} actually discuss the more general case with $p$ replaced by $q \in [1,\infty) \cup \{\infty\}$ on the right-hand sides of \eqref{3.23}, \eqref{3.26}.  

To extend the considerations in Theorem \ref{t3.2} to the case where $H_{\mp,1}$ is replaced by the weighted Hardy operator $H_{\mp,\phi_{\mp},\psi_{\mp}}$ one recalls the following elementary fact, still assuming $F \geq 0$ a.e.~on $(a,b)$.
\begin{align}
& \|H_{-,\phi_-,\psi_-} F\|_{L^p((a,b);vdx)} = \|H_{-,1} (\psi_- F)\|_{L^p((a,b);v \phi_-^pdx)}   \no \\
& \quad = \bigg(\int_a^b dx \, v(x) \phi_-(x)^p \bigg|\int_a^x dx' \, \psi_-(x') F(x')\bigg|^p\bigg)^{1/p}   \no \\
& \quad \leq \wti C_-\bigg(\int_a^b dx \, w(x) \psi_-(x)^{-p} [\psi_-(x) F(x)]^p\bigg)^{1/p}    \no \\
& \quad = \wti C_- \|F\|_{L^p((a,b); wdx)}    \no \\
& \quad = \wti C_- \|\psi_- F\|_{L^p((a,b); w \psi_-^{-p} dx)},     \lb{3.28}
\end{align}
as  well as 
\begin{align}
& \|H_{+,\phi_+,\psi_+} F\|_{L^p((a,b);vdx)} = \|H_{+,1} (\psi_+ F)\|_{L^p((a,b);v \phi_+^pdx)}   \no \\
& \quad = \bigg(\int_a^b dx \, v(x) \phi_+(x)^p \bigg|\int_x^b dx' \, \psi_+(x') F(x')\bigg|^p\bigg)^{1/p}   \no \\
& \quad \leq \wti C_+ \bigg(\int_a^b dx \, w(x) \psi_+(x)^{-p} [\psi_+(x) F(x)]^p\bigg)^{1/p}    \no \\
& \quad = \wti C_+ \|F\|_{L^p((a,b); wdx)}    \no \\
& \quad = \wti C_+ \|\psi_+ F\|_{L^p((a,b); w \psi_+^{-p} dx)},     \lb{3.29}
\end{align}
are equivalent to 
\begin{align}
\big\|H_{-,1} \wti F_-\big\|_{L^p((a,b);v \phi_-^pdx)} 
\leq \wti C_- \big\|\wti F_-\big\|_{L^p((a,b); w \psi_-^{-p} dx)},   \lb{3.30} \\
\big\|H_{+,1} \wti F_+\big\|_{L^p((a,b);v \phi_+^pdx)} 
\leq \wti C_- \big\|\wti F_+\big\|_{L^p((a,b); w \psi_+^{-p} dx)},   \lb{3.31} 
\end{align}
upon identifying $\wti F_{\mp} = \psi_{\mp} F \geq 0$ and replacing the original weights $v$ and $w$ by 
$\wti v = v \phi_{\mp}^p$ and $\wti w = w \psi_{\mp}^{-p}$, respectively. 

Thus, one obtains the following consequence of Theorem \ref{t3.2}, \eqref{3.28}--\eqref{3.31} (see also 
\cite[Theorem~2.3]{KPS17}):

\begin{corollary} \lb{c3.3} 
Assume Hypothesis \ref{h3.1}. \\[1mm]
$(i)$ There exists a constant $\wti C_- \in (0, \infty)$ such that
\begin{equation}
\big(\wti C_-\big)^p \int_a^b dx \, w(x) G(x)^p \geq \int_a^b dx \, v(x) [(H_{-,\phi_-,\psi_-} G)(x)]^p   \lb{3.32}
\end{equation}
for all $G$ measurable on $(a,b)$ and $G \geq 0$~a.e.\ on $(a,b)$, if and only if
\begin{equation}
\wti A_- := \sup_{c \in (a,b)} \bigg(\int_c^b dx \, v(x) \phi_-(x)^p\bigg)^{1/p} 
\bigg(\int_a^c dx \, w(x)^{-p'/p} \psi_-(x)^{p'}\bigg)^{1/p'} < \infty.      \lb{3.33} 
\end{equation}
$($If $p=1$ and hence $p'=\infty$, the second factor in the right-hand side of \eqref{3.33} is interpreted as 
$\|\psi_-/w\|_{L^\infty((a,c);dx)}$.$)$  
Moreover, the smallest constant $\wti C_{0,-} \in (0,\infty)$ in \eqref{3.32} satisfies 
\begin{align}
\begin{split} 
&\wti A_- \leq \wti C_{0,-} \leq p^{1/p} (p')^{1/p'} \wti A_-, \quad p \in (1,\infty),   \\
& \wti C_{0,-} = \wti A_-, \quad p = 1.     \lb{3.34}
\end{split} 
\end{align}
$(ii)$ There exists a constant $\wti C_+ \in (0, \infty)$ such that
\begin{equation}
\big(\wti C_+\big)^p \int_a^b dx \, w(x) G(x)^p \geq \int_a^b dx \, v(x) [(H_{+,\phi_+,\psi_+} G)(x)]^p   \lb{3.35}
\end{equation}
for all measurable $G$ on $(a,b)$ and $G \geq 0$~a.e.\ on $(a,b)$, if and only if
\begin{equation}
\wti A_+ := \sup_{c \in (a,b)} \bigg(\int_a^c dx \, v(x) \phi_+(x)^p\bigg)^{1/p} 
\bigg(\int_c^b dx \, w(x)^{-p'/p} \psi_+(x)^{p'}\bigg)^{1/p'} < \infty.      \lb{3.36}
\end{equation}
$($If $p=1$ and hence $p'=\infty$, the second factor in the right-hand side of \eqref{3.36} is interpreted as 
$\|\psi_+/w\|_{L^\infty((c,b);dx)}$.$)$  
Moreover, the smallest constant $\wti C_{0,+} \in (0,\infty)$ in \eqref{3.35} satisfies 
\begin{align}
\begin{split} 
&\wti A_+ \leq \wti C_{0,+} \leq p^{1/p} (p')^{1/p'} \wti A_+, \quad p \in (1, \infty),    \\
& \wti C_{0,+} = \wti A_+, \quad p = 1.     \lb{3.37}
\end{split} 
\end{align}
\end{corollary}

An application of Corollary \ref{c3.3} then permits one to remove the hypothesis 
$F \in C_0((0,b); \cB)$ in Theorem \ref{t2.2}, Example \ref{e2.3}, and \eqref{2.17} as follows.

\begin{theorem} \lb{t3.4}
Let $p \in [1,\infty)$. \\[1mm] 
$(i)$ In addition to Hypothesis \ref{h2.1}\,$(i)$, assume that $w_j > 0$ a.e.~on $(a,b)$, $j=1,2$, 
and that 
\begin{equation}
F \in L^p\big((a,b); w_1^p [- w_1']^{1-p} w_2^p \, dx; \cB\big).     \lb{3.39}
\end{equation}
Then
\begin{align} 
\begin{split} 
& \int_a^b dx \, w_1(x)^p [- w_1'(x)]^{1-p} w_2(x)^p \|F(x)\|_{\cB}^p   \\
& \quad \geq p^{-p} \int_a^b dx \, [- w_1'(x)] \bigg(\int_a^x dx' \, w_2(x') \|F(x')\|_{\cB}\bigg)^p.   \lb{3.40} 
\end{split}
\end{align}
$(ii)$  In addition to Hypothesis \ref{h2.1}\,$(ii)$, assume that $w_j > 0$ a.e.~on $(a,b)$, 
$j=1,2$, and that 
\begin{equation} 
F \in L^p\big((a,b); w_1^p [w_1']^{1-p} w_2^p \, dx; \cB\big).     \lb{3.41} 
\end{equation} 
Then
\begin{align} 
\begin{split} 
& \int_a^b dx \, w_1(x)^p [w_1'(x)]^{1-p} w_2(x)^p \|F(x)\|_{\cB}^p    \\
& \quad \geq p^{-p} \int_a^b dx \, w_1'(x) \bigg(\int_x^b dx' \, w_2(x') \|F(x')\|_{\cB}\bigg)^p.    \lb{3.42} 
\end{split}
\end{align}
\end{theorem}
\begin{proof}
It suffices to consider item $(i)$ as item $(ii)$ is proved analogously. Identifying 
\begin{align}
G(\, \cdot \,) = \|F(\, \cdot \,)\|_{\cB}, \quad w = w_1^p [- w_1']^{1-p} w_2^p, \quad 
v = [- w_1'], \quad \phi_- = 1, \quad  \psi_- = w_2 
\end{align} 
in Corollary \ref{c3.3}\,$(i)$, the estimate \eqref{3.32} proves boundedness of the weighted Hardy operator 
\begin{equation} 
H_{-,1,w_2} \in \cB\big(L^p\big((a,b); w_1^p [- w_1']^{1-p} w_2^p \, dx\big), 
L^p\big((a,b); [- w_1'] \, dx\big)\big)       \lb{3.45} 
\end{equation}
if and only if 
\begin{align}
\wti A_- & = \sup_{c \in (a,b)} \Bigg[\bigg(\int_c^b dx \, [- w_1'(x)]\bigg)^{1/p}    \no \\
& \quad \times 
\bigg(\int_a^c dx \, \big\{w_1(x)^p \big[- w_1'(x)]^{1-p} w_2(x)^p\big\}^{-p'/p} w_2(x)^{p'}\bigg)^{1/p'}\Bigg]
\no \\
& = \sup_{c \in (a,b)} \Bigg[\bigg(\int_c^b dx \, [- w_1'(x)]\bigg)^{1/p} 
\bigg(\int_a^c dx \, w_1(x)^{-p'} [- w_1'(x)]\bigg)^{1/p'}\Bigg] < \infty,      \lb{3.46}
\end{align}
employing $- p'/p = 1-p'$, $-(1-p)p'/p=1$, temporarily assuming $p \in (1,\infty)$.
The constant $\wti A_-$ is easily estimated and one obtains 
\begin{align}
\wti A_- &= \sup_{c \in (a,b)} \Bigg[[w_1(c) - w_1(b)]^{1/p} 
\bigg[\f{w_1(c)^{1-p'} - w_1(a)^{1-p'}}{p' - 1}\bigg]^{1/p'}\Bigg]   \no \\
& \leq (p' - 1)^{-1/p'} \sup_{c \in (a,b)} \big[w_1(c)^{(1/p) + [(1-p')/p']}\big]    \no \\ 
& = (p' - 1)^{-1/p'} = (p/p')^{1/p'} < \infty, \quad p \in (1,\infty).    \lb{3.47} 
\end{align} 
Thus, $\wti C_- \in (0,\infty)$ as in \eqref{3.32} exists, implying \eqref{3.45}. Given the estimate \eqref{3.47},
the smallest constant $\wti C_{0,-}$ as in \eqref{3.32}, \eqref{3.34} satisfies 
\begin{equation}
\wti C_{0,-} \leq p^{1/p} (p')^{1/p'} \wti A_- \leq p^{1/p} (p')^{1/p'} (p/p')^{1/p'} = p,
\end{equation}
proving the estimate \eqref{3.40}. 

In the case $p=1$, $p' = \infty$, the analog of \eqref{3.46} becomes 
\begin{align}
\wti A_- & = \sup_{c \in (a,b)} \bigg[\bigg(\int_c^b dx \, 
[- w_1'(x)]\bigg) \|1/w_1\|_{L^{\infty}((a,c);dx)}\bigg]   \no \\ 
& = \sup_{c \in (a,b)} \Big[[w_1(c) - w_1(b)] \|1/w_1\|_{L^{\infty}((a,c);dx)}\Big]     \no \\
& = \sup_{c \in (a,b)}\big[ [w_1(c) - w_1(b)] w_1(c)^{-1}\big]      \no \\
& = \sup_{c \in (a,b)}\big[1 - [w_1(b)/w_1(c)]\big]      \no \\
& = \big[1 - [w_1(b)/w_1(a)]\big]  \leq 1, \quad p=1,      
\lb{3.48}
\end{align}
and hence the fact $\wti C_{0,-} = \wti A_-$, according to \eqref{3.33}, also yields \eqref{3.40} for $p=1$. 
\end{proof}

In particular, we now removed the hypothesis $F \in C_0((0,b); \cB)$ in Theorem \ref{t2.2} and replaced it by \eqref{3.39}, \eqref{3.41}. Consequently, this illustrates that Example \ref{e2.3} and 
\eqref{2.17} now extend from $F \in C_0((0,b); \cB)$ to $F \in L^p((a,b); x^{\alpha} dx; \cB)$. 

Due to the fundamental importance of the constants $\wti A_{\mp}$ in connection with smallest constants 
$\wti C_{0,\mp}$ in Hardy-type inequalities (as detailed in \eqref{3.34}, \eqref{3.37}), we now take a second look at them.

\begin{lemma} \lb{l3.5}
Let $p \in [1,\infty)$. \\[1mm] 
$(i)$ Assume Hypothesis \ref{h2.1}\,$(i)$, then
\begin{equation}
\wti A_- = \begin{cases} (p/p')^{1/p'} \Big[1 - \big[w_1(b)^{1/p}/w_1(a)^{1/p}\big]\Big], & p \in (1,\infty), \\[2mm] 
\big[1 - [w_1(b)/w_1(a)]\big], & p = 1.   
 \end{cases}      \lb{3.49}
\end{equation}
In particular, if $w_1(b) = 0$, or $1/w_1(a) = 0$, then 
\begin{equation}
\wti A_- = \begin{cases} (p/p')^{1/p'}, &  p \in (1,\infty), \\[1mm] 
1, & p = 1.   
 \end{cases}      \lb{3.50}
\end{equation}
$(ii)$ Assume Hypothesis \ref{h2.1}\,$(ii)$, then
\begin{equation}
\wti A_+ = \begin{cases} (p/p')^{1/p'} \Big[1 - \big[w_1(a)^{1/p}/w_1(b)^{1/p}\big]\Big], & p \in (1,\infty), \\[2mm] 
\big[1 - [w_1(a)/w_1(b)]\big], & p = 1.   
 \end{cases}      \lb{3.51}
\end{equation}
In particular, if $w_1(a) = 0$, or $1/w_1(b) = 0$, then 
\begin{equation}
\wti A_+ = \begin{cases} (p/p')^{1/p'}, &  p \in (1,\infty), \\[1mm] 
1, & p = 1.   
 \end{cases}      \lb{3.52}
\end{equation}\end{lemma}
\begin{proof}
Again, we prove item $(i)$ only. Starting with the case $p \in (1,\infty)$, we first prove \eqref{3.50} directly (even though that is not necessary). Suppose that $w_1(b) = 0$, then
\begin{align}
\wti A_- &= \sup_{c \in (a,b)} w_1(c)^{p'/(p p')} 
\Bigg[\f{w_1(c)^{1-p'} - w_1(a)^{1-p'}}{p' - 1}\Bigg]^{1/p'}  \no \\
&= (p' - 1)^{-1/p'} \sup_{c \in (a,b)} \Bigg[1 - \f{w_1(c)^{p'-1}}{w_1(a)^{p'-1}}\Bigg]^{1/p'}    \no \\
& = (p' - 1)^{-1/p'} = (p/p')^{1/p'},    \lb{3.53} 
\end{align}
as the supremum is attained for $c =b$. Similarly, if $1/w_1(a) = 0$, then
\begin{align}
\wti A_- &= \sup_{c \in (a,b)} [w_1(c) - w_1(b)]^{1/p} \Bigg[\f{w_1(c)^{1-p'}}{p' - 1}\Bigg]^{1/p'}   \no \\
&= (p' - 1)^{-1/p'} \sup_{c \in (a,b)} \bigg[1 - \f{w_1(b)}{w_1(c)}\bigg]^{1/p}    \no \\
& = (p' - 1)^{-1/p'} = (p/p')^{1/p'},     \lb{3.54}
\end{align}
as the supremum is attained for $c =a$. To deal with the general case \eqref{3.49} (which of course, directly yields \eqref{3.53}, \eqref{3.54}) one can proceed as follows.
\begin{align}
\wti A_- &= \sup_{c \in (a,b)} \Bigg\{[w_1(c) - w_1(b)]^{1/p} 
\Bigg[\f{w_1(c)^{1-p'} - w_1(a)^{1-p'}}{p' - 1}\Bigg]^{1/p'}\Bigg\}   \no \\
&= (p' - 1)^{-1/p'} \sup_{c \in (a,b)} \Bigg\{\bigg[1 - \f{w_1(b)}{w_1(c)}\bigg]^{1/p} 
\bigg[1 - \f{w_1(c)^{p'-1}}{w_1(a)^{p'-1}}\bigg]^{1/p'}\Bigg\}.    \lb{3.55}
\end{align}
To maximize the right-hand side of \eqref{3.55}, we introduce the absolutely continuous function
\begin{equation}
\eta(c) := \bigg[1 - \f{w_1(b)}{w_1(c)}\bigg]^{1/p} 
\bigg[1 - \f{w_1(c)^{p'-1}}{w_1(a)^{p'-1}}\bigg]^{1/p'}, \quad c \in (a,b), 
\end{equation} 
and note that $\eta'(c) = 0$ is equivalent to
\begin{equation}
w_1(c) = w_1(a)^{1/p} w_1(b)^{1/p'}.    \lb{3.57}
\end{equation}
Relation \eqref{3.57} yields a maximum of $\eta$ on $(a,b)$ ($c \in \{a,b\}$ being excluded as a maximum since 
$\eta(a) = \eta(b) =0$ if $p \in (1,\infty)$) and insertion of \eqref{3.57} into the right-hand side of \eqref{3.55} then yields \eqref{3.49} for $p \in (1,\infty)$. 

The case $p=1$ (and hence, $p' = \infty$) follows from
\begin{align}
\wti A_- &= \sup_{c \in (a,b)} \big\{[w_1(c) - w_1(b)] \|1/w_1\|_{L^{\infty}((a,c);dx)}\big\}    \no \\
&= \sup_{c \in (a,b)} \big\{[w_1(c) - w_1(b)]/w_1(c)\big\}   \no \\ 
&= \bigg[1 - \f{w_1(b)}{w_1(a)}\bigg],
\end{align}
as the supremum is attained at $c=a$. 
\end{proof}

\begin{remark} \lb{r3.6}
$(i)$ One observes that the first lines on the right-hand sides of \eqref{3.49}--\eqref{3.52} indeed converge 
to the second lines on the right-hand sides of \eqref{3.49}--\eqref{3.52} as $p \downarrow 1$ and 
$p' \uparrow \infty$. \\[1mm]
$(ii)$ If $w_1(b) \neq 0$ and $1/w_1(a) \neq 0$ (resp., if $w_1(a) \neq 0$ and $1/w_1(b) \neq 0$), then \eqref{3.49} (resp., \eqref{3.51}) proves in conjunction with \eqref{3.34} (resp., \eqref{3.37}) that inequality \eqref{3.40} (resp., \eqref{3.42}), and hence our ad hoc inequality \eqref{2.5} (resp., \eqref{2.6}) is not optimal, that is, the constant $p^{-p}$ in \eqref{3.40} and \eqref{3.42} is not optimal.
\end{remark}

\section{Some Applications to the Operator-Valued Case} \lb{s4}

The principal purpose of this section is to extend Example \ref{e2.3} to the operator-valued situation.

We start with a few preparations. Given a separable, complex Hilbert space $\cH$, we recall that we denote by 
$\cB(\cH)$ the $C^*$-algebra of linear, bounded operators $T \colon \cH \to \cH$ defined on all of $\cH$. 
Similarly, $\cB_p(\cH)$ denote the $\ell^p$-based Schatten--von Neumann trace ideals, $p \in [1,\infty)$. 

The eigenvalues of a bounded linear operator $B \in \cB(\cH)$ are abbreviated by $\lambda_j(B)$, 
$j \in \cJ$, with $\cJ \subseteq \bbN$ an appropriate index set, and the trace of a trace class operator $A \in \cB_1(\cH)$ is denoted by ${\tr}_{\cH} (A)$ and computed via Lidskii's theorem as
\begin{equation}
{\tr}_{\cH} (A) = \sum_{j \in \cJ} \lambda_j(A).     \lb{4.1} 
\end{equation} 
In particular, if $T \in \cB_p(\cH)$ for some $p \in [1,\infty)$, and $|T|$ is defined by $|T| := 
(T^* T)^{1/2}$, one recalls the fact,
\begin{align}
\|T\|_{\cB_p(\cH)}^p = {\tr}_{\cH} \big(|T|^p\big).   \lb{4.2} 
\end{align}
Moreover, if $A : (0,\infty) \to \cB(\cH)$ is weakly measurable, 
$0 \leq A(\dott) \in \cB_1(\cH)$~a.e.~on $(0, \infty)$, and 
${\tr}_{\cH}(A(\dott)) \in L^1((a,b); dt)$, then by an application of the monotone convergence theorem,
\begin{align}
& \bigg\|\int_a^b dt \, A(t)\bigg\|_{\cB_1(\cH)} = {\tr}_{\cH} \bigg(\int_a^b dt \, A(t)\bigg)    \no \\
& \quad = \sum_{n \in \cN} \int_a^b dt \, (e_n, A(t) e_n)_{\cH} 
= \int_a^b dt \, \sum_{n \in \cN} (e_n, A(t) e_n)_{\cH}    \lb{4.3} \\
& \quad = \int_a^b dt \, {\tr}_{\cH} (A(t)) = \int_a^b dt \, \|A(t)\|_{\cB_1(\cH)},     \no 
\end{align} 
where $\{e_n\}_{n \in \cN}$ represents a complete orthonormal system in $\cH$, with $\cN \subseteq \bbN$ an appropriate index set. In this context we also recall the well-known fact,
\begin{equation}
\bigg\|\int_a^b dt \, A(t)\bigg\|_{\cB_p(\cH)} 
\leq \int_a^b dt \, \|A(t)\|_{\cB_p(\cH)}, \quad 
p \in [1,\infty),    \lb{4.4}
\end{equation}
and similarly with $\cB_p(\cH)$ replaced by $\cB(\cH)$. 

Given these preparations, one can restate Example \ref{e2.3} in the case where $\cB = \cB_p(\cH)$ as follows.

\begin{corollary} \lb{c4.1}
Let $p \in [1,\infty)$, $b \in(0, \infty) \cup \{\infty\}$, and suppose that 
$F : (0,b) \to \cB(\cH)$ is a weakly measurable map satisfying 
$\|F(\dott)\|_{\cB_p(\cH)} \in L^p((0,b); x^{\alpha} dx)$, with $\alpha \in \bbR$ chosen according to items $(i)$ and $(ii)$ below: \\[1mm]  
$(i)$ If $\alpha < p-1$, then \eqref{2.13} implies
\begin{equation}
{\tr}_{\cH}\bigg(\int_0^b dx \, \, x^{\alpha} \, |F(x)|^p\bigg) \geq 
\bigg(\f{|\alpha - p + 1|}{p}\bigg)^p {\tr}_{\cH}\bigg(\int_0^b dx \, x^{\alpha -p} 
\bigg|\int_0^x dx' \, F(x')\bigg|^p\bigg).    \lb{4.5} 
\end{equation} 
$(ii)$  If $\alpha > p - 1$, then \eqref{2.14} implies
\begin{equation}
{\tr}_{\cH}\bigg(\int_0^b dx \, \, x^{\alpha} \, |F(x)|^p\bigg) \geq 
\bigg(\f{|\alpha - p + 1|}{p}\bigg)^p {\tr}_{\cH}\bigg(\int_0^b dx \, x^{\alpha -p} 
\bigg|\int_x^b dx' \, F(x')\bigg|^p\bigg).    \lb{4.6} 
\end{equation} 
In both cases $(i)$ and $(ii)$, the constant $[(|\alpha - p + 1|)/p]^p$ is best possible and equality holds if and only 
if $F = 0$~a.e.~on $(0,b)$. 
\end{corollary}
\begin{proof}
It suffices to consider item $(i)$. Then an application of \eqref{4.1}--\eqref{4.4} yields 
\begin{align}
& \bigg(\f{|\alpha - p + 1|}{p}\bigg)^p {\tr}_{\cH}\bigg(\int_0^b dx \, x^{\alpha -p} 
\bigg|\int_0^x dx' \, F(x')\bigg|^p\bigg)     \no \\
& \quad = \bigg(\f{|\alpha - p + 1|}{p}\bigg)^p \int_0^b dx \, x^{\alpha -p} 
{\tr}_{\cH}\bigg(\bigg|\int_0^x dx' \, F(x')\bigg|^p\bigg)   \quad (\text{by \eqref{4.3}})  \no \\
& \quad = \bigg(\f{|\alpha - p + 1|}{p}\bigg)^p \int_0^b dx \, x^{\alpha - p} 
\bigg\|\int_0^x dx' \, F(x') \bigg\|_{\cB_p(\cH)}^p    \quad (\text{by \eqref{4.2}}) \no \\
& \quad \leq \bigg(\f{|\alpha - p + 1|}{p}\bigg)^p 
\int_0^b dx \, x^{\alpha - p} \bigg(\int_0^x dx' \, \|F(x')\|_{\cB_p(\cH)}\bigg)^p   \quad (\text{by \eqref{4.4}}) \no \\
& \quad \leq \int_0^b dx \, x^{\alpha} \|F(x)\|_{\cB_p(\cH)}^p    \quad (\text{by \eqref{2.13}}) \no \\
& \quad = \int_0^b dx \, x^{\alpha} {\tr}_{\cH}\big(|F(x)|^p\big)    \quad (\text{by \eqref{4.2}}) \no \\
& \quad = {\tr}_{\cH}\bigg(\int_0^b dx \, x^{\alpha} |F(x)|^p\bigg)   \quad (\text{by \eqref{4.3}}).   \lb{4.7}
\end{align}
The final part about optimality of the constant on the right-hand side in \eqref{4.5} and \eqref{4.6}, and the equality part, then follow as in Example \ref{e2.3}.
\end{proof}

We note that the case $\alpha = 0$, $b = \infty$, $F(\dott) \geq 0$~a.e.~on 
$(0,\infty)$ in \eqref{4.5} was proved by 
Hansen \cite[Theorem~2.4]{Ha09} on the basis of a convexity argument (see also \cite{HKPP10}, \cite{Ki18}). Our strategy of proof is different and based on that in Theorem \ref{t2.2}. 

Next, following Hansen \cite{Ha09}, we will remove the trace in Corollary \ref{c4.1} in the case where $p \in [1,2]$. 

We start by recalling \cite[Lemma~2.1]{Ha09}: 

\begin{lemma} \lb{l4.2} 
Let $p \in [1,2]$, and suppose that $F : (0,\infty) \to \cB(\cH)$ is a weakly measurable map satisfying 
$F(\dott) \geq 0$~a.e.~on $(0,\infty)$, and $\int_0^{\infty} dx \, x^{-1} F(x)^p \in \cB(\cH)$. Then, 
\begin{align}
\int_0^{\infty} dx \, x^{-1} F(x)^p \geq \int_0^{\infty} dx \, x^{-1 - p} \bigg(\int_0^x dx' \, F(x')\bigg)^p.  \lb{4.8} 
\end{align}
The constant $1$ on the right-hand side of the inequality \eqref{4.8} is best possible. 
\end{lemma}

Employing Lemma \ref{l4.2} we can prove the principal result of this section.

\begin{theorem} \lb{t4.3}
Let $p \in [1,2]$, and suppose that $F : (0,\infty) \to \cB(\cH)$ is a weakly measurable map satisfying 
$F(\dott) \geq 0$~a.e.~on $(0,\infty)$, and 
$\int_0^{\infty} dx \, x^{\alpha} F(x)^p \in \cB(\cH)$, with $\alpha \in \bbR$ chosen according to items $(i)$ and $(ii)$ below: \\[1mm] 
$(i)$ If $\alpha < p-1$, then 
\begin{equation}
\int_0^{\infty} dx \, \, x^{\alpha} \, F(x)^p \geq 
\bigg(\f{|\alpha - p + 1|}{p}\bigg)^p \int_0^{\infty} dx \, x^{\alpha -p} 
\bigg(\int_0^x dx' \, F(x')\bigg)^p.    \lb{4.9} 
\end{equation} 
$(ii)$  If $\alpha > p - 1$, then 
\begin{equation}
\int_0^{\infty} dx \, \, x^{\alpha} \, F(x)^p \geq 
\bigg(\f{|\alpha - p + 1|}{p}\bigg)^p \int_0^{\infty} dx \, x^{\alpha -p} 
\bigg(\int_x^{\infty}  dx' \, F(x')\bigg)^p.    \lb{4.10} 
\end{equation} 
In both cases $(i)$ and $(ii)$, the constant $[(|\alpha + 1 - p|)/p]^p$ is best possible and equality holds if and only 
if $F = 0$~a.e.~on $(0,b)$. 
\end{theorem}
\begin{proof} We start by proving item $(i)$. 
Closely following the strategy of proof in \cite[Theorem~2.3]{Ha09}, we introduce
\begin{equation}
G(x) = F\big(x^{p/|\alpha - p + 1|}\big) x^{(1+\alpha)/|\alpha - p + 1|}, \quad x > 0,
\end{equation}
and the change of variables
\begin{equation}
y = x^{p/|\alpha - p + 1|}, \quad 
dy = [p/|\alpha - p + 1|] x^{(1 + \alpha)/|\alpha - p + 1|} dx. 
\end{equation}
Then Lemma \ref{l4.2} applied to $G$ yields
\begin{align}
\int_0^{\infty} dx \, x^{-1} G(x)^p 
&= \int_0^{\infty} dx \, x^{-1} F\big(x^{p/|\alpha - p + 1|}\big)^p 
x^{p(1+\alpha)/|\alpha - p + 1|}
\no \\
& \geq \int_0^{\infty} dx \, x^{-1-p} \bigg(\int_0^x dx' G(x')\bigg)^p  
\quad (\text{by \eqref{4.8}}) \no \\
& = \int_0^{\infty} dx \, x^{-1-p} \bigg(\int_0^x dt \, 
F\big((t^{p/|\alpha - p + 1|}\big) t^{(1+\alpha)/|\alpha - p + 1|}\bigg)^p     \no \\
& = \bigg(\f{|\alpha - p + 1|}{p}\bigg)^p \int_0^{\infty} dx \, x^{-1-p} 
\bigg(\int_0^{x^{p/|\alpha - p + 1|}} dy \, F(y)\bigg)^p.    \lb{4.13} 
\end{align}
Introducing another change of variables
\begin{equation}
w = x^{p/|\alpha - p + 1|}, \quad dw \, w^{-1} = [p/|\alpha - p + 1|] dx \, x^{-1},      \lb{4.14} 
\end{equation}
then implies
\begin{align}
& \bigg(\f{|\alpha - p + 1|}{p}\bigg) \int_0^{\infty} dw \, w^{\alpha} F(w)^p     \no \\
&\quad = \bigg(\f{|\alpha - p + 1|}{p}\bigg) \int_0^{\infty} dw \, w^{-1} F(w)^p w^{1 + \alpha}     \no \\
& \quad = 
\int_0^{\infty} dx \, x^{-1} F\big(x^{p/|\alpha - p + 1|}\big)^p x^{p(1+\alpha)/|\alpha - p + 1|}     \no \\
& \quad \geq \bigg(\f{|\alpha - p + 1|}{p}\bigg)^{p} \int_0^{\infty} dx \, x^{-1-p} 
\bigg(\int_0^{x^{p/|\alpha - p + 1|}} dy \, F(y)\bigg)^p  \quad (\text{by \eqref{4.13}})     \no \\
& \quad = \bigg(\f{|\alpha - p + 1|}{p}\bigg)^{p+1} \int_0^{\infty} dw \, 
w^{-1 - (p-1-\alpha)} 
\bigg(\int_0^w dy \, F(y)\bigg)^p  \quad (\text{by \eqref{4.14}})\no \\
& \quad = \bigg(\f{|\alpha - p + 1|}{p}\bigg)^{p+1} \int_0^{\infty} dw \, 
w^{\alpha - p} \bigg(\int_0^w dy \, F(y)\bigg)^p, 
\end{align}
proving \eqref{4.9}. 

While $\Phi(F) = x^{-1} \int_0^x dx' \, F(x')$ represents a positive, unital map 
(i.e., $\Phi(F) \geq 0$ if $F \geq 0$ and $\Phi(I_{\cH}) = I_{\cH}$), $\int_x^{\infty} dx' \, F(x')$ cannot possibly 
be of this type and hence one cannot simply follow the proof of \cite[Theorem~2.3]{Ha09} to derive \eqref{4.10}. Fortunately, the following elementary alternative approach applies. Introducing the change of variables, 
\begin{equation}
y = 1/x, \quad G(y) = F(1/y) y^{-2}, 
\end{equation}   
in \eqref{4.9} (w.r.t.~$x$ on either side in \eqref{4.9} and, especially, w.r.t.~$x'$ on the right-hand side of 
\eqref{4.9}) results in
\begin{align}
\int_0^{\infty} dy \, y^{\beta} G(y)^p \geq \bigg(\f{|\beta - p + 1|}{p}\bigg)^p \int_0^{\infty} dx \, x^{\beta - p} 
\bigg(\int_x^{\infty} dy \, G(y)\bigg)^p,
\end{align}
where
$\beta = 2p -2 - \alpha$, and hence $\alpha < p-1$ is equivalent to $\beta > p-1$.  

The final part about optimality of the constant on the right-hand side in \eqref{4.9} and \eqref{4.10}, and the equality part, then follow as in Corollary \ref{c4.1} from Example \ref{e2.3} upon taking the trace on either 
side of \eqref{4.9} and \eqref{4.10}.
\end{proof}

Again we note that the case $\alpha = 0$ in \eqref{4.9} was proved by Hansen \cite[Theorem~2.3]{Ha09}; 
he also proved that Theorem \ref{t4.3} does not extend to $p > 2$. 

While we focused on the underlying interval $(0,\infty)$ in Theorem \ref{t4.3}, the analogous case $(0,b)$, 
$b \in (0,\infty)$ follows upon employing the variable transformations discussed in \cite[p.~36--38]{KPS17}.  

Extending the definition of $(H_{\mp,\ell} F)(x)$, $x \in (0,\infty)$, $\ell \in \bbN$, in \eqref{2.15}, \eqref{2.16} to the operator-valued context where $F : (0,\infty) \to \cB(\cH)$ is a weakly measurable map satisfying 
$F(\dott) \geq 0$~a.e.~on $(0,\infty)$, and for all $c \in (0,\infty)$, 
$\int_0^c dx \, F(x)^p \in \cB(\cH)$ in connection with $H_{-,\ell}$ and 
$\int_c^{\infty} dx \, F(x)^p \in \cB(\cH)$ in connection with $H_{+,\ell}$, the facts \eqref{4.9}, \eqref{4.10} can 
be rewritten as 
\begin{align}
\begin{split} 
& \int_0^{\infty} dx \, x^{\gamma - p} [H_{\mp,1} (F(\dott))(x)]^p  \\
& \quad \leq \bigg(\f{p}{|\gamma - p +1|}\bigg)^p \int_0^{\infty} dx \, x^{\gamma} F(x)^p, 
\quad p \in [1, \infty), \; \gamma \lessgtr p-1.   \lb{4.18} 
\end{split} 
\end{align}
Thus one obtains the following result.

\begin{corollary} \lb{c4.4} 
Let $p \in [1,2]$, and suppose that $F : (0,\infty) \to \cB(\cH)$ is a weakly measurable map satisfying 
$F(\dott) \geq 0$~a.e.~on $(0,\infty)$, and 
$\int_0^{\infty} dx \, x^{\alpha} F(x)^p \in \cB(\cH)$, with $\alpha \in \bbR$ chosen according to \eqref{4.19} below. Then
\begin{align}
&  \int_0^{\infty} dx \, x^{\alpha} F(x)^p    \no \\
& \quad \geq \prod_{k=1}^{\ell} \bigg(\f{|\alpha - kp + 1|}{p}\bigg)^p \int_0^{\infty} dx \, x^{\alpha - \ell p} 
[H_{\mp,\ell} (F(\, \cdot \,))(x)]^p,   \lb{4.19} \\
& \hspace*{5.35cm}  \alpha \lessgtr \begin{cases} p - 1, \\  \ell p - 1, 
\end{cases} \ell \in \bbN.   \no 
\end{align} 
\end{corollary}
\begin{proof}
It suffices to iterate \eqref{4.18} by applying it to appropriate $F = F_{\ell}$ as in the derivation of \eqref{2.17}. 
\end{proof}

Replacing $F \geq 0$ by $|F| = (F^* F)^{1/2}$ and mimicking the differential version of the Hardy inequalities at the end of Section \ref {s2} yields 
\begin{align}
\begin{split} 
\int_0^{\infty} dx \, x^{\alpha} |f'(x)|^p \geq \bigg(\f{|\alpha - p +1|}{p}\bigg)^p  
\int_0^{\infty} dx \, x^{\alpha - p} |f(x)|^p,& \\
p \in [1,2], \;  \alpha \in \bbR, \; f \in C_0^{\infty}((0,\infty); \cB(\cH)).&    \lb{4.20}
 \end{split}
\end{align}
Iterating \eqref{4.20} then yields as in \eqref{2.24}
\begin{align}
\begin{split} 
\int_0^{\infty} dx \, x^{\alpha} |f^{(n)}(x)|^p \geq \f{\prod_{j=1}^k |\alpha- jp +1|^p}{p^{kp}}
\int_0^{\infty} dx \, x^{\alpha - kp} |f^{(n-k)}(x)|^p,& \\ 
p \in [1,2], \; 1 \leq k \leq n, \; n \in \bbN, \;  \alpha \in \bbR, \; f \in C_0^{\infty}((0,\infty); \cB(\cH)).&    \lb{4.21} 
\end{split} 
\end{align}

\medskip

\noindent {\bf Acknowledgments.} We are indebted to Paul Hagelstein for very helpful discussions. We 
also gratefully acknowledge a variety of helpful suggestions by the anonymous referee and thank him 
for a critical reading of our manuscript.


\end{document}